\documentclass[journal,twoside,web]{ieeecolor}
\usepackage{lcsys}

\usepackage{etoolbox}
\makeatletter
\@ifundefined{color@begingroup}%
  {\let\color@begingroup\relax
   \let\color@endgroup\relax}{}%
\def\fix@ieeecolor@hbox#1{%
  \hbox{\color@begingroup#1\color@endgroup}}
\patchcmd\@makecaption{\hbox}{\fix@ieeecolor@hbox}{}{\FAILED}
\patchcmd\@makecaption{\hbox}{\fix@ieeecolor@hbox}{}{\FAILED}

\usepackage[scaled=0.9]{helvet}
\usepackage[T1]{fontenc}

\usepackage{graphicx}
\usepackage[colorlinks=true,citecolor=cyan,linkcolor=red,urlcolor=cyan]{hyperref}
\usepackage{amsmath,amssymb,amsthm}
\usepackage{mathrsfs}
\usepackage{bm}
\usepackage{url}
\usepackage{newtxmath}
\usepackage{booktabs}

\newtheorem{proposition}{Proposition}
\newtheorem{theorem}{Theorem}
\newtheorem{remark}{Remark}
\allowdisplaybreaks
\begin{document}

\title{Conformal Margins for Electrical Transmission Capacity Under Fixed Balancing Policies}
\author{Sushanth Balaraman$^{\star,\dagger}$, Shriram Srinivasan$^\dagger$, Kaarthik Sundar$^\dagger$
\thanks{$^\star$ Dept. of Computer Science, Texas A\&M University, College Station, TX}
\thanks{$^\dagger$ Los Alamos National Laboratory, Los Alamos, NM}
\thanks{Corresponding author: \texttt{kaarthik@lanl.gov}}
}

\maketitle
\thispagestyle{empty}
\begin{abstract}
This article presents a data-driven approach to minimally adjust forecast-based steady-state generation dispatch in electrical transmission networks. The goal is to manage congestion caused by uncertainty in demand and renewable generation under a fixed automatic generation control (AGC) policy operating without saturation. Under this policy, we first show that the post-AGC residual injection distribution is invariant to the scheduled generation dispatch. This result separates line loading into a dispatch-dependent baseline and an uncertainty-induced residual component. Using historical forecast errors, we then apply split conformal prediction to construct transmission-line margins with distribution-free, finite-sample linewise or joint coverage guarantees. We also calibrate the aggregate pre-AGC mismatch to determine the generation headroom required for the prescribed AGC response. We incorporate these calibrated quantities into a convex secure re-dispatch problem that minimally adjusts the nominal schedule while reserving sufficient transmission capacity and AGC headroom. The resulting corrected dispatch provides finite-sample probabilistic guarantees on per-line security and unsaturated AGC operation. Case studies on the RTS-GMLC system demonstrate substantial reductions in out-of-sample line-limit violations with modest operating-cost increases.
\end{abstract}

\begin{IEEEkeywords}
Conformal split calibration, singular value decomposition, congestion modes, automatic generation control, uncertainty
\end{IEEEkeywords}

\section{Introduction} \label{sec:intro}
Electric transmission network congestion, arising when power flows approach or exceed line limits, remains a persistent challenge for power grid operators because scheduled generation and demand rarely coincide with their real-time realizations \cite{pillay2015congestion}. Consequently, net-injection uncertainty and balancing actions can overload lines that are secure under the nominal dispatch computed by a deterministic optimal power flow (OPF), rendering the system insecure under plausible forecast errors. 
Stochastic, robust, and chance-constrained OPF formulations address the aforementioned issue by incorporating uncertainty directly into scheduling \cite{roald2023power}. These approaches typically follow a one-step paradigm, jointly optimizing the nominal dispatch and balancing response under uncertainty while enforcing transmission constraints. However, their modeling and computational demands can hinder deployment in large-scale operational workflows. System operators therefore often rely on scalable deterministic dispatch procedures \cite{zhao2023viability}, while real-time imbalances are managed through Automatic Generation Control (AGC), which allocates the aggregate mismatch among generators according to prescribed participation factors \cite{jaleeli1992understanding}. Although AGC restores power balance, it does not explicitly enforce transmission limits and can therefore create or exacerbate congestion.

This article develops a two-step alternative under a direct-current (DC) power-flow model. The first step retains the operator’s existing deterministic dispatch, while the second applies a lightweight corrective layer that minimally re-dispatches generation to mitigate uncertainty-induced congestion and reserve sufficient headroom for a fixed, unsaturated AGC policy. The key observation is that, with fixed participation factors and exogenous forecast errors, the post-AGC residual distribution is invariant to the candidate re-dispatch, separating dispatch-dependent baseline loading from uncertainty-induced deviations. This allows transmission-line safety margins and aggregate-mismatch bounds, which determine the required AGC headroom, to be calibrated independently of the corrected schedule. These quantities are then enforced through tightened line and generation-capacity constraints in a convex secure re-dispatch problem. Thus, the method integrates into existing dispatch workflows without modifying the underlying OPF formulation or solver. We adopt the DC model to isolate the proposed conformally calibrated framework, leaving extension to linearized alternating-current (AC) models for future work.

The contributions of this article are threefold. First, we show that, under a fixed unsaturated AGC policy and exogenous forecast errors, the post-AGC residual distribution is invariant to the candidate re-dispatch. Second, we use split conformal prediction \cite{pmlr-v128-vovk20a} to obtain distribution-free, finite-sample linewise or joint coverage guarantees for transmission-line deviations, and to calibrate the aggregate-mismatch range required to reserve sufficient AGC headroom, i.e., the generation capacity that must remain available to accommodate the prescribed AGC response. Third, we incorporate the resulting transmission margins and AGC-headroom requirements into a convex secure re-dispatch problem that minimally adjusts the nominal schedule while providing finite-sample probabilistic guarantees on line security and unsaturated AGC operation.

\section{Mathematical Model} \label{sec:model}
\subsection{Network Loading Map and Congestion Energy}
We model the transmission network as a connected graph $\mathcal G=(\mathcal N,\mathcal E)$ with $n$ buses and $m$
transmission lines. After assigning an arbitrary orientation to each line, let $A\in\mathbb R^{m\times n}$ denote the line-bus incidence matrix, $B \triangleq \operatorname{diag}(b_\ell)$ denote the diagonal matrix of lines' susceptance, and define the weighted Laplacian $L \triangleq A^\top B A$. For a connected $\mathcal G$, $\operatorname{null}(L)$ equals the $\operatorname{span}\{\bm 1\}$. Under the lossless steady-state DC power-flow model, the net nodal injection vector $\bm p\in\mathbb R^n$ is balanced, i.e., $\bm 1^\top\bm p=0$. The bus voltage phase angles and line flows are therefore given, up to an arbitrary reference angle, by 
\begin{gather}
    \bm\theta=L^\dagger\bm p,
    \qquad
    \bm f=BA L^\dagger\bm p,
\end{gather}
where $L^\dagger$ is the Moore--Penrose pseudo-inverse of $L$. Let $f_\ell^{\max}$ denote the positive thermal limit of line $\ell$, and define $D \triangleq \operatorname{diag}(1/f_{\ell}^{\max})$. Then, the net nodal injection-to-line-loading map is defined as 
\begin{equation}
    K\triangleq D B A L^\dagger,
    \qquad
    \bm \ell(\bm p)\triangleq K\bm p.
    \label{eq:congestion-linear-map}
\end{equation}
We summarize the network-wide transmission loading by the scalar
\begin{equation}
    \mathscr E(\bm p)
    \triangleq 
    \sum_{\ell\in\mathcal E}
    \left(\frac{f_\ell}{f_\ell^{\max}}\right)^2
    =
    \bm p^\top K^\top K\bm p = \|\bm\ell(\bm p)\|_2^2.
    \label{eq:congestion-energy}
\end{equation}
We refer to $\mathscr E$ as the \emph{congestion energy}. It is a dimensionless measure of the aggregate squared line-loading norm, with each line weighted by its thermal capacity. A line operating at its limit contributes a value of one to $\mathscr E$, while an overloaded line contributes a value greater than one. Although $\mathscr E$ is not a substitute for enforcing individual line limits, it provides a smooth scalar measure of overall thermally normalized transmission stress. 
\subsection{Balancing Policy and Residual Injection} \label{subsec:agc}
Let $\bar{\bm p}\in\mathbb R^n$ denote the forecast-based net-injection vector obtained by solving a deterministic DC-OPF.  $\bar{\bm p}$ satisfies the nodal power-balance condition $\bm 1^\top\bar{\bm p}=0$ and the nominal operating constraints included in the DC-OPF. 
Let $\bm\omega\in\mathbb R^n$ denote the net-injection forecast error realized after $\bar{\bm p}$ is determined. 
Before balancing, the realized net-injection is $\bar{\bm p}+\bm\omega$, where $\bm 1^\top\bm\omega$ is the aggregate power imbalance.
We consider an AGC policy with participation vector $\bm\alpha\in\mathbb R_+^n$, satisfying $\bm 1^\top\bm\alpha=1$. 
The participation vector is either specified separately or determined as part of the deterministic DC-OPF and is held fixed when the forecast error is realized.
In response to the aggregate imbalance, AGC applies the balancing adjustment
\begin{gather}
    \Delta\bm p^{\mathrm{AGC}} =-\bm\alpha\left(\bm 1^\top\bm\omega\right), \label{eq:agc}
\end{gather}
and post-balancing injection is
\begin{gather}
    \bm p^{\mathrm{actual}} =
    \bar{\bm p} + \bm\omega
    -\bm\alpha \left(\bm 1^\top\bm\omega\right) =
    \bar{\bm p}+A_{\alpha}\bm\omega,
    \label{eq:actual-injection} \\
    \text{where } A_{\alpha} \triangleq I-\bm\alpha\bm 1^\top .
    \label{eq:projection}
\end{gather}
Since $\bm 1^\top A_{\alpha}=0$, AGC removes the aggregate mismatch and ensures
$\bm 1^\top\bm p^{\mathrm{actual}}=0$.

We define the \emph{post-AGC residual injection} as:
\begin{equation}
    \bm r
    \triangleq
    \bm p^{\mathrm{actual}}-\bar{\bm p}
    = A_{\alpha}\bm\omega .
    \label{eq:agc-residual}
\end{equation}
Thus, $\bm 1^\top\bm r=0$ for every realization. Decomposing the residual into its mean and fluctuating components yields 
\begin{gather}
    \bm r=\bm\mu_r+\tilde{\bm r},
    \quad
    \bm\mu_r\triangleq\mathbb E[\bm r],
    \quad
    \mathbb E[\tilde{\bm r}]=0 
    \quad \Sigma_{\tilde r} \triangleq \mathbb E[\tilde{\bm r} \tilde{\bm r}^\top].
    \label{eq:residual-decomposition}
\end{gather}
Accordingly, the actual line-loading vector is
\begin{equation}
    \bm \ell(\bm p^{\mathrm{actual}})
    = K(\bar{\bm p}+\bm r)
     = K(\bar{\bm p}+\bm\mu_r) +K\tilde{\bm r}.
    \label{eq:actual-line-loading}
\end{equation}
Eq.~\eqref{eq:actual-line-loading} separates the deterministic baseline loading from the centered uncertainty-induced deviation.

\begin{proposition}[Set-point invariance under fixed AGC] \label{prop:invariance}
Suppose the distribution of the primitive forecast error $\bm\omega$ is exogenous to the scheduled injection and that the same fixed, unsaturated AGC policy $A_{\alpha}$ applies before and after a change in $\bar{\bm p}$. Then the distribution of $\bm r$, and consequently $\bm\mu_r$ and $\Sigma_{\tilde r}$, is invariant to  $\bar{\bm p}$.
\end{proposition}
\begin{proof}
For every realization,
$\bm r=\bm p^{\mathrm{actual}}-\bar{\bm p}
=A_{\alpha}\bm\omega$. Since $A_{\alpha}$ is fixed and the distribution of $\bm \omega$ is independent of $\bar{\bm p}$, the distribution of $\bm r$ is unchanged by the choice of $\bar{\bm p}$.
\end{proof}
\begin{remark}
The invariance may fail if the AGC response saturates, the participation factors or generator commitment change, the forecast-error distribution or $K$ depends on the nominal schedule, as may occur under an operating-point-dependent AC linearization.
\end{remark}
\subsection{Residual Congestion Modes} \label{subsec:modes}
Under the fixed, unsaturated AGC policy, \eqref{eq:agc-residual}--\eqref{eq:residual-decomposition} yield
the realized congestion energy
\begin{gather}
    \mathscr{E}(\bm p^{\mathrm{actual}}) = \|\bm \ell(\bm p^{\mathrm{actual}})\|_2^2. \label{eq:congenstion-energy-realization} 
\end{gather}
Consequently, 
\begin{gather}
        \mathbb E [\mathscr{E}(\bm p^{\mathrm{actual}})] = \underbrace{ \|K(\bar{\bm p} + \bm \mu_r)\|_2^2 }_{\text{baseline loading}} + 
\underbrace{ \operatorname{tr}(K\Sigma_{\tilde r} K^\top)}_{\text{residual fluctuation energy}} \label{eq:split}
\end{gather}
where the cross-term vanishes because $\tilde{\bm r}$ is zero-mean. The first term depends on the candidate set-point $\bar{\bm p}$ and the residual mean $\bm \mu_r$, whereas Proposition~\ref{prop:invariance} implies that $\Sigma_{\tilde r}$, and hence the second term, remains invariant under re-dispatch. 

We characterize the uncertainty-induced component through the singular value decomposition  (SVD)
\begin{gather}
    K\Sigma_{\tilde r}^{1/2} = USV^\top, \label{eq:svd}
\end{gather}
where $ S = \operatorname{diag}(s_1, \dots, s_q),$ and $s_1\geqslant \dots \geqslant s_q \geqslant 0$. It follows that
\begin{gather}
    K\Sigma_{\tilde r} K^\top = U S^2 U^\top ~~\Rightarrow~~ \operatorname{tr}(K\Sigma_{\tilde r} K^\top) = \sum_{k=1}^q s_k^2 \label{eq:trace}
\end{gather}
The decomposition in \eqref{eq:svd} is equivalent to the Karhunen--Lo\`eve (KL) expansion~\cite{alexanderian2026primer} of the uncertainty-induced line-loading deviation $K\tilde{\bm r}$. In power-system uncertainty studies, KL expansions are commonly applied in the input space so that retained directions capture dominant variability in uncertain renewable generation or demand before network propagation \cite{safta2016efficient}. Here, we perform the decomposition in the line-loading output space instead. Indeed,
\begin{gather}
\operatorname{Cov}(K\tilde{\bm r})
=
K\Sigma_{\tilde r}K^\top
=
US^2U^\top.
\label{eq:cov}
\end{gather}
Thus, the columns of $U$, i.e., the left singular vectors of \eqref{eq:svd}, are also eigenvectors of the uncertainty-induced line-loading covariance; we refer to them as \textit{residual congestion modes}. The resulting modes therefore prioritize directions by their contribution to line-loading variability, rather than to variability in the primitive forecast errors alone. The quantity $s_k^2$ measures the contribution of mode $k$ to the expected residual fluctuation energy in \eqref{eq:split}.


In practice, $\bm\mu_r$ and $\Sigma_{\tilde r}$ are empirically estimated using a training data set $\mathcal D_{\mathrm{tr}}$ with cardinality $N_{\mathrm{tr}}$ as
\begin{gather}
\widehat{\bm \mu}_r
=
\frac{1}{N_{\mathrm{tr}}}
\sum_{i=1}^{N_{\mathrm{tr}}} \bm r^{(i)}, \label{eq:mean} \\
\widehat{\Sigma}_{\tilde r}
=
\frac{1}{N_{\mathrm{tr}}-1}
\sum_{i=1}^{N_{\mathrm{tr}}}
\left(\bm r^{(i)}-\widehat{\bm \mu}_r\right)
\left(\bm r^{(i)}-\widehat{\bm \mu}_r\right)^{\top}. \label{eq:sigma}
\end{gather}
The empirical congestion modes are then obtained as the SVD of
\begin{equation}
K\widehat{\Sigma}_{\tilde r}^{1/2}
=
\widehat{U}\widehat{S}\widehat{V}^{\top}. \label{eq:empirical-svd}
\end{equation}
The congestion-mode decomposition provides an interpretable representation of the covariance structure of the residual line-loading deviations and can also support reduced-order approximations in large networks (see Remark~\ref{rem:low-rank}). The conformal constructions developed in the subsequent section, however, operate
directly on the observed residual line-loading deviations and do not require truncation of this decomposition.

\section{Split Conformal Calibration} \label{sec:conformal-margins}
In this section, we use split conformal prediction \cite{pmlr-v128-vovk20a} to calibrate the uncertain quantities required for secure re-dispatch in Section~\ref{sec:redispatch}. We first construct margins for the post-AGC line-loading deviations under two distinct coverage objectives. \emph{Linewise calibration} provides a marginal coverage guarantee for each transmission line separately and therefore does not require normalization across lines. In contrast, \emph{joint calibration} provides a simultaneous network-wide guarantee by forming a common nonconformity score across all lines. Since the magnitudes of uncertainty-induced loading deviations can vary significantly across lines and between positive and negative directions, the joint construction uses training-derived one-sided scales to normalize these deviations before aggregation. In both constructions, the training data set $\mathcal D_{\mathrm{tr}}$ is used to estimate the residual mean $\widehat{\bm{\mu}}_r$ using \eqref{eq:mean}, while an independent calibration data set $\mathcal D_{\mathrm{cal}}$ is used for conformal calibration. We also use the calibration data to construct a conformal interval for the aggregate pre-AGC power mismatch, which determines the generation headroom to reserve for the prescribed AGC response.

\subsection{Linewise Coverage} \label{subsec:linewise}
For any residual realization $\bm r$, define the centered uncertainty-induced line-loading deviation
\begin{gather}
    \bm y(\bm r)  \triangleq K\left(\bm r-\widehat{\bm\mu}_r\right), \qquad y_{\ell}(\bm r) = \left[ \bm y(\bm r)\right]_{\ell} \label{eq:deviations-raw}
\end{gather}
For sample $i$, we write $\bm y^{(i)}=\bm y(\bm r^{(i)})$ and $y_\ell^{(i)}= y_{\ell}(\bm r^{(i)})$.
For each calibration residual $\bm r^{(i)} \in \mathcal D_{\mathrm{cal}}$, we define the upper- and lower-tail, line-specific nonconformity scores directly from the corresponding line-loading excursions, 
\begin{gather}
R_{i,\ell}^{+} \triangleq [y_\ell^{(i)}]_+
~~\text{ and }~~
R_{i,\ell}^{-} \triangleq [-y_\ell^{(i)}]_+.\label{eq:non-conformity-scores}
\end{gather}
where $[x]_+ \triangleq \max\{x,0\}$. Since a separate conformal quantile is computed for every line and direction, these scores require no additional line-dependent normalization. Let 
\begin{gather*}
R_{(1),\ell}^{+} \leqslant\cdots\leqslant R_{(N_{\mathrm{cal}}),\ell}^{+} 
~\text{ and }~ R_{(1),\ell}^{-} \leqslant\cdots\leqslant R_{(N_{\mathrm{cal}}),\ell}^{-} 
\end{gather*}
denote the ordered calibration scores for each line. For prescribed tail violation probabilities $\epsilon^+$ and $\epsilon^-$, define 
\begin{gather}
        k^\pm \triangleq
    \left\lceil
    (N_{\mathrm{cal}}+1)(1-\epsilon^\pm)
    \right\rceil,
    \qquad
    q_{1-\epsilon^\pm}^{\pm}(\ell)
    \triangleq
    R_{(k^\pm),\ell}^{\pm},
    \label{eq:conformal-quantiles}
\end{gather}
with $R_{(N_{\mathrm{cal}}+1),\ell}^{\pm}=+\infty$. Then the corresponding linewise upper and lower margins are given by 
\begin{gather}
    m_{\ell}^+ = q^+_{1-\epsilon^+}(\ell) ~~\text{ and }~~ m_{\ell}^- = q^-_{1-\epsilon^-}(\ell) \label{eq:linewise-margins}
\end{gather}
Now, for any future or test residual realization $\bm r$, define the line-$\ell$ coverage event
\begin{gather}
\mathcal M_{\ell}(\bm r) \triangleq
\left\{ -m_\ell^- \leqslant y_{\ell}(\bm r) \leqslant m_\ell^+ \right\}.
\label{eq:margin-event}
\end{gather}
Under exchangeability \cite{pmlr-v128-vovk20a} of the calibration residuals and a future residual $\bm r^{\mathrm{new}}$, the following proposition holds:
\begin{proposition}[Finite-sample linewise coverage]
\label{p erop:linewise-coverage}
Condition on the training data $\mathcal D_{\mathrm{tr}}$, and suppose that the residual mean $\widehat{\bm\mu}_r$ is computed using only $\mathcal D_{\mathrm{tr}}$. Let $\{\bm r^{(i)}\}_{i=1}^{N_{\mathrm{cal}}}$ denote the calibration residuals, and suppose that they and a future residual $\bm r^{\mathrm{new}}$ are exchangeable conditional on $\mathcal D_{\mathrm{tr}}$. If $\epsilon^++\epsilon^-=\epsilon$, then the margins in \eqref{eq:linewise-margins} satisfy
\begin{equation}
\mathbb P\left( \mathcal M_{\ell}(\bm r^{\mathrm{new}})
\,\middle|\,
\mathcal D_{\mathrm{tr}}
\right) \geqslant 1-\epsilon, \quad \forall \ell \in \mathcal E.
\label{eq:conditional-linewise-coverage}
\end{equation}
\end{proposition}
\begin{proof}
Conditioned on $\mathcal D_{\mathrm{tr}}$, the residual mean $\widehat{\bm \mu}_r$ is fixed. Exchangeability of the calibration and future residuals therefore induces exchangeability of the corresponding upper- and lower-tail nonconformity scores. The split-conformal rank argument \cite{pmlr-v128-vovk20a} yields
\begin{gather*}
\mathbb P\left( R_{\mathrm{new},\ell}^{\pm}> q_{1-\epsilon^{\pm}}^{\pm}(\ell)
\,\middle|\,
\mathcal D_{\mathrm{tr}} \right) \leqslant \epsilon^\pm.
\end{gather*}
The two events correspond to violations of the upper and lower bounds defining $\mathcal M_{\ell}(\bm r^{\mathrm{new}})$. The union bound therefore yields 
$$\mathbb P\left( \mathcal M_{\ell}(\bm r^{\mathrm{new}})^c
\,\middle|\,
\mathcal D_{\mathrm{tr}}
\right) \leqslant \epsilon^+ + \epsilon^- = \epsilon ~~ \forall \ell \in \mathcal E$$ where $\mathcal M_{\ell}(\bm r^{\mathrm{new}})^c$ denotes the complement of \eqref{eq:margin-event}, completing the proof.
\end{proof}
The guarantee is distribution-free and finite-sample, i.e., no parametric model for the residual distribution is required. In this setting, exchangeability requires that the calibration and future residuals arise under the same forecasting regime and fixed AGC policy. 

\subsection{Joint Coverage with One-sided Normalization} 
Linewise calibration guarantees coverage for each transmission line individually, but does not ensure simultaneous coverage across all lines. To obtain a network-wide guarantee, we instead define a single nonconformity score for each residual realization based on the largest deviation across the network. Directly taking the maximum of the raw deviations in \eqref{eq:deviations-raw}, however, would compare lines with different levels of uncertainty-induced variability and would result in a common absolute loading margin without accounting for their relative variability. We therefore introduce the positive and negative one-sided scales $a_{\ell}^+$ and $a_{\ell}^-$, computed exclusively from the training data as
\begin{equation}
\begin{gathered}
    a_{\ell}^{\pm}
    =
    \left(
        \eta^2+
        \frac{1}{N_{\mathrm{tr}}-1}
        \sum_{i=1}^{N_{\mathrm{tr}}}
        \left[\pm ~y_{\ell}^{(i)}\right]_{+}^{2}
    \right)^{1/2}.
    \label{eq:directional_scales} 
\end{gathered}
\end{equation}
Here, $\eta>0$ prevents degenerate scales. These quantities normalize the line-loading excursions before taking the maximum across lines; they determine how the joint margin is allocated across lines, while the subsequent conformal quantile determines its probabilistic magnitude.  
\begin{proposition}[Finite-sample joint coverage]
\label{prop:joint-conformal}
Define the one-sided joint nonconformity scores
\begin{gather*}
    R_i^{\mathrm{joint}+}
    \triangleq
    \max_{\ell\in\mathcal E}
    \left\{
        \frac{[y_\ell^{(i)}]_+}{a_\ell^+},
    \right\}, ~~
    R_i^{\mathrm{joint}-}
    \triangleq
    \max_{\ell\in\mathcal E}
    \left\{
        \frac{[-y_\ell^{(i)}]_+}{a_\ell^-}
    \right\},
\end{gather*}
Let $q_{1-\epsilon^\pm}^{\mathrm{joint}\pm}$ be their split-conformal quantiles defined analogous to \eqref{eq:conformal-quantiles}, and set
\begin{gather*}
    m_\ell^+
    =
    q_{1-\epsilon^+}^{\mathrm{joint}+}a_\ell^+,
    \qquad
    m_\ell^-
    =
    q_{1-\epsilon^-}^{\mathrm{joint}-}a_\ell^-.\label{eq:joint-margins}
\end{gather*}
If $\epsilon = \epsilon^+ + \epsilon^-$, then under exchangeability conditional on $\mathcal D_{\mathrm{tr}}$, the above margins satisfy
\begin{gather*}
    \mathbb P\left(
        \bigcap_{\ell \in \mathcal E} \mathcal M_{\ell}(\bm r^{\mathrm{new}}) \,|\, \mathcal D_{\mathrm{tr}}
    \right)
    \geqslant 1-\epsilon. \label{eq:guarantee-joint}
\end{gather*}
\end{proposition}
The proof is analogous to that of Proposition \ref{p erop:linewise-coverage} and is omitted. Since joint calibration controls the maximum excursion across all lines simultaneously, it is generally more conservative than separate linewise calibration; the degree of conservatism depends on the number of lines and the dependence structure of their deviations.

\begin{remark}[Low-rank approximation] \label{rem:low-rank}
For large networks, a rank-$t$ approximation of \eqref{eq:empirical-svd} may be used to represent the dominant residual congestion modes, with $t$ chosen to capture a prescribed fraction of the empirical residual fluctuation energy. The corresponding projection $\bm y_t^{(i)} = \widehat U_t\widehat U_t^\top \bm y^{(i)}$ provides a rank-$t$ approximation of the centered line-loading deviation. The conformal guarantees above, however, are based on the original deviations $\bm y^{(i)}$.
\end{remark}

\subsection{Calibration of AGC Headroom} \label{subsec:AGC-headroom}
In addition to transmission-line margins, secure re-dispatch must reserve sufficient AGC headroom for the prescribed balancing response. For each calibration sample $i$, define the aggregate pre-AGC mismatch as
\begin{gather}
    \delta^{(i)} \triangleq \bm 1^\top \bm\omega^{(i)}.  \label{eq:aggregate-mismatch}
\end{gather}
Applying split-conformal calibration to $\{\delta^{(i)}\}_{i=1}^{N_{\mathrm{cal}}}$ yields an interval $[\widehat{\delta}^{\mathrm{low}}, \widehat{\delta}^{\mathrm{up}}]$ such that, under exchangeability,
$$ \mathbb P(\mathcal H^\mathrm{new} \,|\,
        \mathcal D_{\mathrm{tr}}) \geqslant 1-\epsilon_{\mathrm H} \text{ where, } 
\mathcal H^\mathrm{new} \triangleq \left\{\widehat{\delta}^{\mathrm{low}} \leqslant \delta^\mathrm{new} \leqslant \widehat{\Delta}^{\mathrm{up}} \right\}.$$
Since the AGC adjustment corresponding to a mismatch $\Delta$ is $-\bm\alpha\Delta$, these bounds determine the AGC headroom that must be reserved for the fixed policy. Typically, $\widehat{\delta}^{\mathrm{low}}<0< \widehat{\delta}^{\mathrm{up}}$: the upper mismatch bound determines the required downward headroom, while the lower bound determines the required upward headroom. We use these calibrated bounds in the next section to obtain the corrected dispatch.

\section{Secure Re-dispatch} \label{sec:redispatch}

In Section~\ref{sec:conformal-margins}, historical data was used to quantify the uncertainty the operating point must accommodate while leaving the nominal $\bar{\bm p}$ unchanged.
In this section, we use the previously calibrated quantities, namely the estimated residual mean $\widehat{\bm\mu}_r$, the calibrated line-loading margins $\bm m^+$ and $\bm m^-$, and the aggregate-mismatch interval $[\widehat{\delta}^{\mathrm{low}}, \widehat{\delta}^{\mathrm{up}}]$ required to reserve AGC headroom to now determine a corrected schedule.

Starting from the nominal schedule $\bar{\bm p}$, we introduce a preventive re-dispatch adjustment $\Delta\bm p$ and define
\begin{gather}
    \bar{\bm p}^{c} \triangleq \bar{\bm p}+\Delta\bm p. \label{eq:corrected-schedule}
\end{gather}
Here, $\Delta\bm p$ is a preventive re-dispatch decision made before the future forecast error is realized; it is distinct from the random
post-AGC residual $\bm r$ in \eqref{eq:agc-residual}. For a future realization $\bm r^{\mathrm{new}}$, the post-AGC injection is therefore $\bar{\bm p}^{c}+\bm r^{\mathrm{new}}$. 
During re-dispatch, the smallest balanced adjustment $\Delta\bm p$ is chosen such that the corrected schedule leaves sufficient transmission capacity for the calibrated residual deviations and sufficient generation headroom for the prescribed AGC response. Since the residual distribution is invariant to the candidate re-dispatch under Proposition~\ref{prop:invariance}, the calibrated uncertainty quantities remain applicable while optimizing the corrected schedule.

For compactness, $\Delta \bm p$ represents controllable adjustments at the bus level. Let $\bm p^{\min}$ and $\bm p^{\max}$ denote the bus-level
limits obtained by aggregating the controllable generation limits at each bus. At buses without controllable resources, $\Delta p_j=0$. The secure re-dispatch is obtained as the solution of the following convex quadratic program:
\begin{subequations}
\label{eq:secure-redispatch}
\begin{flalign}
    \min\quad &\|\Delta\bm p\|_2^2 \label{eq:redispatch-objective}\\
    \mathrm{s.t.}\quad &\bm 1^\top\Delta\bm p=0, \label{eq:redispatch-balance}\\
    &\bm p^{\min}
    +\bm\alpha\widehat{\delta}^{\mathrm{up}}
    \leqslant
    \bar{\bm p}+\Delta\bm p
    \leqslant
    \bm p^{\max}
    +\bm\alpha\widehat{\delta}^{\mathrm{low}},
    \label{eq:redispatch-headroom}\\
    &-\bm 1+\bm m^- \leqslant
    K\left( \bar{\bm p}+\Delta\bm p +\widehat{\bm\mu}_r \right)
    \leqslant
    \bm 1-\bm m^+. \label{eq:redispatch-line-constraints}
\end{flalign}
\end{subequations}
\eqref{eq:redispatch-balance} preserves nodal power balance, while \eqref{eq:redispatch-headroom} ensures sufficient upward and downward headroom for the prescribed aggregate-mismatch range. The tightened line constraints \eqref{eq:redispatch-line-constraints} account for the estimated residual mean and reserve the calibrated margins around the corrected baseline loading. The following result formalizes how these deterministic reserve constraints convert the calibrated uncertainty events into security guarantees for the corrected schedule.

\begin{theorem}[Linewise security guarantee for the corrected dispatch]
\label{thm:coverage-preserving-redispatch}
Let $\Delta\bm p^\star$ be any feasible solution of \eqref{eq:secure-redispatch}, and define 
\begin{gather*}
    \bar{\bm p}^{c,\star}
    \triangleq
    \bar{\bm p}+\Delta\bm p^\star.
\end{gather*}
Suppose the residual model satisfies the fixed-policy assumptions of Proposition~\ref{prop:invariance}, and let the calibrated margins satisfy
\begin{gather*}
    \mathbb P\left(\mathcal M_{\ell}(\bm r^{\mathrm{new}}) 
    \,\middle|\,
    \mathcal D_{\mathrm{tr}} \right) \geqslant 1-\epsilon
\end{gather*}
for each line $\ell \in \mathcal E$. Let the calibrated AGC-headroom interval satisfy
\begin{gather*}
    \mathbb P(\mathcal H^\mathrm{new} 
    \,|\,
    \mathcal D_{\mathrm{tr}} ) \geqslant 1-\epsilon_{\mathrm H}.
\end{gather*}
Then on $\mathcal M_{\ell}(\bm r^{\mathrm{new}}) \cap \mathcal H^{\mathrm{new}}$, the prescribed AGC response remains unsaturated and
\begin{gather}
    -1 \leqslant \left[ K\left(\bar{\bm p}^{c,\star} +\bm r^{\mathrm{new}} \right) \right]_\ell \leqslant 1, 
    \label{eq:redispatch-security-guarantee}
\end{gather}
Consequently, for every line $\ell \in \mathcal E$
\begin{gather*}
\mathbb P(\text{$\ell$ is secure and AGC remains unsaturated}) \geqslant 1- \epsilon_{\mathrm{H}} - \epsilon.
\end{gather*}
\end{theorem}
\begin{proof}
Feasibility of \eqref{eq:redispatch-line-constraints} implies
\begin{gather*}
    -1+m_\ell^-
    \leqslant
    \left[ K\left( \bar{\bm p}^{c,\star} +\widehat{\bm\mu}_r \right) \right]_\ell
    \leqslant
    1-m_\ell^+.
    \label{eq:proof-tightened-line}
\end{gather*}
For event $\mathcal M_{\ell}(\bm r^{\mathrm{new}})$, we have
\begin{gather*}
    -m_\ell^-
    \leqslant
    \left[ K\left( \bm r^{\mathrm{new}} -\widehat{\bm\mu}_r \right) \right]_\ell
    \leqslant
    m_\ell^+.
    \label{eq:proof-residual-event}
\end{gather*}
Adding the two inequalities yields \eqref{eq:redispatch-security-guarantee}. For the headroom part, feasibility of \eqref{eq:redispatch-headroom} yields
\begin{gather*}
    \bm p^{\min}
    +\bm\alpha\widehat{\delta}^{\mathrm{up}}
    \leqslant
    \bar{\bm p}^{c,\star}
    \leqslant
    \bm p^{\max}
    +\bm\alpha\widehat{\delta}^{\mathrm{low}}.
\end{gather*}
Since $\bm\alpha\geqslant\bm 0$, on $\mathcal H^{\mathrm{new}}$,
\begin{gather*}
    \bm p^{\min} \leqslant
    \bar{\bm p}^{c,\star} -\bm\alpha\delta^{\mathrm{new}}
    \leqslant \bm p^{\max}.
\end{gather*}
Hence the prescribed AGC response remains unsaturated on $\mathcal H^{\mathrm{new}}$. Finally, using union bound, we have for every $\ell \in \mathcal E$ $$\mathbb P(\mathcal M_{\ell}(\bm r^{\mathrm{new}}) \cap \mathcal H^{\mathrm{new}} \,|\, \mathcal D_{\mathrm{tr}}) \geqslant 1- \epsilon_{\mathrm{H}} - \epsilon;$$ no independence assumption is required.
\end{proof}

\section{Results}  \label{sec:results}
All computational experiments use the RTS-GMLC test system~\cite{barrows2019ieee}, comprising $73$ buses, $120$ branches, $158$ generators, and $51$ loads. We obtain load and renewable-generation forecasts and realizations from the associated RTS-GMLC time-series data. We select one high-solar forecast hour per July day, yielding 31 operating points treated as a common forecasting regime where exchageability holds; distribution shift and temporal non-exchangeability are outside scope of the current article. For each point, a deterministic DC-OPF gives the nominal dispatch, with $12$ real-time realizations per point ($372$ samples) split chronologically into training, calibration, and test sets of $120$, $120$, and $132$. 
The source code, RTS-GMLC data used in the experiments, balancing policies, and proposed algorithms are available at \texttt{\url{https://github.com/kaarthiksundar/PowerGridConformalMargins}}. 

\subsection{Empirical Coverage: Joint vs. Linewise Calibration}
We calibrate both linewise and joint conformal margins using the training and calibration sets, with nominal coverage levels $(1-\epsilon)\in\{0.90,0.95\}$ and an equal allocation between the upper and lower tails, $\epsilon^+=\epsilon^-=\epsilon/2$. Their out-of-sample performance is evaluated on the $N_{\mathrm{te}}=132$ test residuals. To distinguish the two calibration procedures, let $\mathcal M_{\ell,s}^{(i)}$ denote the line-$\ell$ margin event for test residual $i$, where $s\in\{\mathrm J,\mathrm L\}$ denotes joint or linewise calibration, respectively:
\begin{gather*}
\mathcal M_{\ell,s}^{(i)} \triangleq \left\{
-m_{\ell,s}^- \leqslant 
\left[ K\left( \bm r^{(i)}-\widehat{\bm\mu}_r \right) \right]_\ell \leqslant
m_{\ell,s}^+
\right\}.
\end{gather*}
For the jointly calibrated margins, we report the empirical joint coverage
\begin{gather*}
p_{\mathrm J} = \frac{1}{N_{\mathrm{te}}} \sum_{i=1}^{N_{\mathrm{te}}}
\mathbf 1 \left\{ \bigcap_{\ell\in\mathcal E} \mathcal M_{\ell,\mathrm J}^{(i)}
\right\}.
\end{gather*}
For the linewise margins, we report the average empirical marginal coverage
\begin{gather*}
p_{\mathrm L} = \frac{1}{|\mathcal E|} \sum_{\ell\in\mathcal E} \frac{1}{N_{\mathrm{te}}}
\sum_{i=1}^{N_{\mathrm{te}}} \mathbf 1
\left\{ \mathcal M_{\ell,\mathrm L}^{(i)} \right\},
\end{gather*}
as well as the empirical joint coverage obtained by requiring all linewise margin events to hold:
\begin{gather*}
p_{\mathrm L}^{\cap} = \frac{1}{N_{\mathrm{te}}} \sum_{i=1}^{N_{\mathrm{te}}}
\mathbf 1 \left\{ \bigcap_{\ell\in\mathcal E} \mathcal M_{\ell,\mathrm L}^{(i)} \right\}.
\end{gather*}
We report $p_{\mathrm L}^{\cap}$ only for comparison, since independently calibrated linewise margins provide marginal, rather than joint network-wide, coverage guarantees.

The Table~\ref{tab:empirical-coverage} reports the resulting test-set coverage frequencies. The jointly calibrated margins achieve empirical joint coverage of $87.12\%$ and $95.45\%$ at nominal levels of $90\%$ and $95\%$, respectively. In contrast, intersecting independently calibrated linewise events yields joint coverage of only $5.30\%$ and $6.82\%$. This illustrates that high marginal coverage on individual lines does not imply joint coverage across the network. The reported test-set frequencies should not be interpreted as exact realizations of the nominal coverage levels, since the conformal guarantees are probabilistic statements under exchangeability rather than deterministic guarantees for a fixed finite test set.
\begin{table}[t]
\centering
\caption{Empirical test-set coverage under joint and linewise conformal calibration. All entries are percentages.}
\label{tab:empirical-coverage}
\begin{tabular}{cccc}
\toprule
$1-\epsilon$ &
$p_{\mathrm J}$ &
$p_{\mathrm L}$ &
$p_{\mathrm L}^{\cap}$  \\
\midrule
$90$ & $87.12$ & $85.71$ & $5.30$ \\
$95$ & $95.45$ & $89.91$ & $6.82$  \\
\bottomrule
\end{tabular}
\end{table}

Figure~\ref{fig:line-margin-histograms} compares the distributions across transmission lines of the calibrated upper and lower margins obtained from the joint and linewise procedures. 
As expected, joint calibration produces wider, more heterogeneous margins because it enforces simultaneous network-wide coverage.
For some branches, the jointly calibrated margins consume a large fraction of the line-capacity range. Recall that the tightened feasible interval for line $\ell$ is $\left[-1+m_{\ell}^{-},1-m_{\ell}^{+}\right]$.
Hence, if $m_{\ell}^{-}+m_{\ell}^{+}>2$, this interval is empty regardless of the chosen re-dispatch.  Even when this condition is not reached, large joint margins can substantially restrict the feasible corrected schedule. This conservatism reflects the stronger requirement imposed by a distribution-free joint network-wide coverage certificate.
\begin{figure}[htbp]
    \centering
    \includegraphics[scale=1.1]{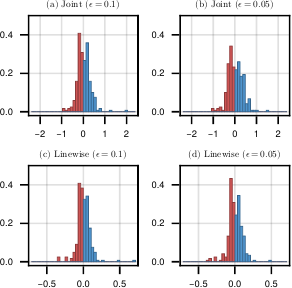}
    \caption{Distribution across transmission lines of the calibrated upper and lower line-loading margins under joint and linewise conformal calibration. Joint calibration provides simultaneous network-wide coverage but generally yields more conservative margins than separate linewise calibration.}
    \label{fig:line-margin-histograms}
\end{figure}

\subsection{Out-of-Sample Security and Re-dispatch Cost}
We use linewise margins for the re-dispatch study because joint margins can substantially restrict or eliminate re-dispatch feasibility. We evaluate out-of-sample performance using $1-\epsilon=0.90$ and AGC-headroom coverage $1-\epsilon_{\mathrm H}=0.95$ on the $132$ held-out real-time realizations. 

Figure~\ref{fig:original} shows the line-loadings under the original forecast-based dispatch. Residual uncertainty produces multiple thermal-limit violations, with the largest overload reaching approximately $15\%$ above the corresponding limit. After applying the secure re-dispatch, Figure~\ref{fig:new} shows that only one violation remains across the test set, with an exceedance below $0.3\%$. Consistent with this reduction, the corrected dispatch achieves an average empirical linewise security rate of $99.99\%$ and an empirical simultaneous line-security rate of $99.24\%$, with all lines satisfying their limits in $131$ of the $132$ test realizations. The AGC-adjusted injections also satisfy the prescribed resource limits in all test realizations, yielding $100\%$ empirical coverage. The observed line-security rates can exceed the nominal conformal margin level because the calibrated margin event is sufficient, but not necessary, for satisfying the thermal limits; excursions beyond a margin need not cause an overload when additional transmission headroom is available.

These security improvements are achieved with modest economic impact. Although \eqref{eq:redispatch-objective} minimizes squared set-point movement rather than generation cost, the secure re-dispatch increases operating cost by only $1.46\%$ on average, with a maximum increase of $7.79\%$. Thus, the proposed secure re-dispatch substantially reduces out-of-sample congestion while largely preserving the original deterministic dispatch.

\begin{figure}
    \centering
    \includegraphics[scale=1.1]{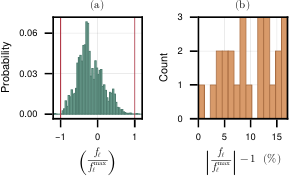}
    \caption{Line loadings under the original forecast-based dispatch. (a) Test-set line-flow distribution with thermal limits at $\pm 1$. (b) Relative exceedance of violated line limits; the maximum overload is approximately $15\%$.}
    \label{fig:original}
\end{figure}

\begin{figure}
    \centering
    \includegraphics[scale=1.1]{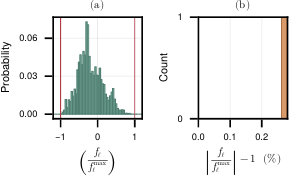}
    \caption{Empirical test-set line-loadings after secure re-dispatch with $1-\epsilon=0.90$ and $1-\epsilon_{\mathrm H}=0.95$. Only one line-limit violation remains, with an exceedance below $0.3\%$.}
    \label{fig:new}
\end{figure}

\section{Conclusion} \label{sec:conclusion}
This article presents a data-driven secure re-dispatch framework to mitigate uncertainty-induced transmission congestion while preserving the operator's deterministic dispatch procedure. Under a fixed unsaturated AGC policy and exogenous forecast errors, the post-AGC residual distribution is invariant to the candidate re-dispatch, enabling split conformal calibration of post-AGC line-loading deviations and aggregate mismatch independently of the corrected schedule. The resulting transmission margins and AGC-headroom requirements are incorporated into a convex secure re-dispatch problem that minimally adjusts the nominal schedule while providing probabilistic protection against line overloads and AGC saturation. Computational experiments on the RTS-GMLC network show substantial reductions in out-of-sample line-limit violations with modest operating-cost increases. Future work will consider AC linearization and changing or saturated balancing policies.

\bibliography{refs.bib}
\bibliographystyle{IEEEtran}
\end{document}